\documentclass[12pt]{amsart}
\usepackage{amssymb}
\usepackage{verbatim}

\numberwithin{equation}{section}

\newtheorem{question}[equation]{Question}
\newtheorem{thm}[equation]{Theorem}

\newtheorem{cor}[equation]{Corollary}
\newtheorem{prop}[equation]{Proposition}

\theoremstyle{remark}

\newtheorem{remark}[equation]{Remark}

\theoremstyle{definition}

\newcommand{\C}{{\mathbb C}}
\newcommand{\Z}{{\mathbb Z}}

\newcommand{\Q}{{\mathbb Q}}

\begin{document}



\title{Two questions on polynomial decomposition}

\author{Brian K. Wyman}
\address{
  Department of Mathematics,
  University of Michigan,
  Ann Arbor, MI 48109--1043,
  USA
}
\curraddr{
PNYLAB, LLC,
902 Carnegie Center, Suite 200,
Princeton, NJ 08540--6530,
USA
}
\email{brian@pnylab.com}

\author{Michael E. Zieve}
\address{
  Department of Mathematics,
  University of Michigan,
  Ann Arbor, MI 48109--1043,
  USA
}
\email{zieve@umich.edu}
\urladdr{www.math.lsa.umich.edu/$\sim$zieve/}

\thanks{The authors thank Mel Hochster and Alex Mueller for valuable discussions.
The authors were partially supported by the NSF under grants
DMS-0502170 and DMS-0903420, respectively.}

\begin{abstract}
Given a univariate polynomial $f(x)$ over a ring $R$, we examine when
we can write $f(x)$ as $g(h(x))$ where $g$ and $h$ are polynomials of degree
at least $2$.
We answer two questions of Gusi\'c regarding when the existence of
such $g$ and $h$ over an extension of $R$ implies the existence of
such $g$ and $h$ over $R$.
\end{abstract}


\maketitle


\section{Introduction}

Let $R$ be a ring.  If $f(x)\in R[x]$ has degree at least $2$, we say that
$f$ is \textit{decomposable} (over $R$) if we can write $f(x)=g(h(x))$
for some nonlinear $g,h\in R[x]$; otherwise we say $f$ is indecomposable.
Many authors have studied decomposability of polynomials in case
$R$ is a field (see, for instance, \cite{BWZ, BN, BM, DW, Engstrom, Fried-Ritt,
FriedMacRae, Levi, McConnell, Ritt, Schinzel,
Tortrat, Zannier, ZM}).
The papers \cite{DG,DT,Gusic} examine decomposability over more general
rings, in the wake of the following result of Bilu and Tichy \cite{BT}:
for $f,g\in R[x]$, where $R$ is the ring of $S$-integers of a number field,
if the equation $f(u)=g(v)$ has infinitely many solutions $u,v\in R$ then
$f$ and $g$ have decompositions of certain types.
In the present note we answer two questions
on this topic posed recently by Gusi\'c \cite{Gusic}:

\begin{question}\label{q1}
Prove or disprove.  Let $R$ be an integral domain of zero characteristic.
Let $S$ denote the integral closure of $R$ in the field of fractions
of $R$.  Assume that $S\ne R$.  Then there exists a monic polynomial
$f$ over $R$ that is decomposable over $S$ but not over $R$.
\end{question}

\begin{question}\label{q2}
Prove or disprove.  Let $R$ be the ring of integers of a number field $K$.
Assume that $R$ is not a unique factorization domain.  Then there exists
a polynomial $f$ over $R$ that is decomposable over $K$ but not over $R$.
\end{question}

The most significant difference between these questions is that the first
question addresses monic polynomials, while the second addresses
arbitrary polynomials.

We will show that the first question has a negative answer, and the second
has a positive answer.  We also pose two new questions along similar lines.

These questions were motivated by two results due to
Turnwald~\cite[Prop.\ 2.2 and 2.4]{Turnwald}, which assert that if
$R$ is an integral domain of characteristic zero, and $K$ is a field
containing $R$, then:
\begin{enumerate}
\item If $R$ is integrally closed in its field of fractions, then every
indecomposable monic polynomial over $R$ is indecomposable over~$K$.
\item If $R$ is a unique factorization domain,
then every indecomposable polynomial over $R$ is indecomposable over $K$.
\end{enumerate}

The special case $R=\Z$ of Turnwald's first result was first proved by
Wegner \cite[p.~9]{Wegner}, and was later rediscovered in \cite[Thm.~2]{DG}.
Both of Turnwald's results were rediscovered in \cite[Thm.~2.1 and 2.5]{Gusic}.

Further results about polynomial decomposition over rings appear in the
first author's thesis \cite{Wyman} and in forthcoming joint papers by
the authors.


\section{Monic polynomials}

In this section we show that Question~\ref{q1} has a negative answer.
We prove this by means of the following result.

\begin{prop}
Let $S$ be an integral domain of characteristic zero, and let $R$ be a
subring of $S$.  If monic $g,h\in xS[x]$  satisfy $g(h(x))\in R[x]$,
then $g,h\in (\Q.R)[x]$.
\end{prop}

\begin{proof}
Write $g=\sum_{i=1}^n g_i x^i$ and $h=\sum_{i=1}^m h_j x^j$, with
$g_n=h_m=1$.  Then, for $1\le k<m$, the coefficient of $x^{nm-k}$ in
$g(h(x))$ is $nh_{m-k}$ plus a polynomial (with integer coefficients)
in $h_{m-k-1}, h_{m-k-2},\dots,h_{m-1}$.  Since this coefficient lies in $R$,
it follows by induction on $k$ that each $h_{m-k}$ lies in $\Q.R$.
Likewise, for $1\le k<n$, the coefficient
of $x^{nm-km}$ in $g(h(x))$ equals the sum of $g_{n-k}$ and a polynomial
(with integer coefficients) in $g_{n-k+1},g_{n-k+2},\dots,g_{n-1}, h_1,h_2,\dots,h_{m-1}$.
Since this coefficient lies in $R$, induction on $k$ implies that $g_{n-k}$ lies
in $\Q.R$, as desired.
\end{proof}

\begin{cor}\label{c}
Let $S$ be an integral domain of characteristic zero, and let $R$ be
a subring of $S$ such that $(\Q.R)\cap S=R$.  Then every indecomposable
monic polynomial over $R$ is indecomposable over $S$.
\end{cor}

\begin{proof}
Let $f\in R[x]$ be a monic polynomial which is decomposable over $S$.
Say $f=G(H(x))$ where $G,H\in S[x]$ are nonlinear.  Denoting the leading
coefficients of $G$ and $H$ by $u$ and $v$, we compute the leading coefficient
of $f$ as $1=uv^{\deg(G)}$.  Now let $g=G(vx+H(0))-f(0)$ and $h=uv^{\deg(G)-1}(H(x)-H(0))$,
so $g$ and $h$ are nonlinear monic polynomials in $xS[x]$ such that
$g(h(x))=f(x)-f(0)$ lies in $R[x]$.  By the previous result, $g$ and $h$
have coefficients in $\Q.R$; since they also have coefficients in $S$, in fact
their coefficients lie in $(\Q.R)\cap S=R$, so $f$ is decomposable over $R$.
\end{proof}

We now exhibit an explicit example showing that Question~\ref{q1} has a
negative answer.  In light of the above corollary, it suffices to exhibit
an integral domain $R$ of characteristic zero whose integral closure $S$
satisfies $S\ne R$ and $(\Q.R)\cap S=R$.  One example is $R=\Z[t^2,t^3]$,
where $t$ is transcendental over $\Q$.  The field of fractions of $R$ is
$\Q(t)$, and the integral closure of $R$ in $\Q(t)$ is $S:=\Z[t]$, so
indeed $S\ne R$ and $(\Q.R)\cap S=R$. \qed

\vspace{0.3cm}
In view of Corollary~\ref{c} (and Turnwald's result), we pose the following
modified version of Question~\ref{q1}:

\begin{question} \label{us1}
Let $R$ be an integral domain of characteristic zero, and let $S$ be the
integral closure of $R$ in its field of fractions.  If $(\Q.R)\cap S\ne R$,
then does there exist an indecomposable monic polynomial over $R$ which
decomposes over $S$?
\end{question}

\begin{remark}
If  $R$  is a subring of a number field $K$, then $\Q.R=K$; hence,
for such rings, Question~\ref{us1} reduces to Question~\ref{q1}.
It would be interesting to know whether these questions have an
affirmative answer in this case.
\end{remark}


\section{Non-monic polynomials}

In this section we show that Question 1.2 has a positive answer.

\begin{thm} \label{main}
If $R$ is the ring of integers of a number field $K$, and $R$ is not a
unique factorization domain, then there exists an indecomposable
polynomial over $R$ which decomposes over $K$.
\end{thm}

In fact we prove the following more general result.

\begin{thm} \label{main2}
Let $R$ be an integral domain which contains an element having two
inequivalent factorizations into irreducibles, and suppose that every
nonsquare in $R$ remains a nonsquare in the fraction field $K$ of $R$.
Then there is an indecomposable degree-$4$ polynomial over $R$ which
decomposes over $K$.
\end{thm}

Recall that two factorizations into irreducibles are \textit{inequivalent}
if there is no bijective correspondence between the irreducibles in the
first and the irreducibles in the second such that corresponding irreducibles
are unit multiples of one another.

\begin{proof}[Proof that Theorem~\ref{main2} implies Theorem~\ref{main}]
Let $R$ be the ring of integers of a number field $K$, and suppose that
$R$ is not a unique factorization domain.  By induction on the norm,
every element of $R$ which is neither zero nor a unit can be written as
the product of irreducible elements.  Thus, since $R$ is not a unique
factorization domain, $R$ must contain an element which has two
inequivalent factorizations into irreducibles.

Let $u$ be an element of $R$ which is a square in $K$.  Then the polynomial
$x^2-u$ has a root in $K$, but this is a monic polynomial over $R$ so its
roots are integral over $R$; hence these roots lie in $R$ since $R$ is
integrally closed in $K$.
\end{proof}

\begin{proof}[Proof of Theorem~\ref{main2}]
Pick an element of $R$ having two inequivalent factorizations into
irreducibles.  By repeatedly removing
irreducibles from the first factorization which have a unit multiple
in the second factorization, we obtain an element $\alpha\in R\setminus
(\{0\}\cup R^*)$ having two factorizations into irreducibles such that
no irreducible in the first factorization has a unit multiple in the
second factorization.  Let $\ell$ be an irreducible in the first
factorization, and write the second factorization as $p_1\dots p_r$
where no $p_i$ is a unit multiple of $\ell$.  Letting $s$ be the least
positive integer for which $\ell\mid p_1\dots p_s$, it follows that
$a:=p_1\dots p_{s-1}$ is an element of $R$ such that $\ell\mid ap_s$
but $\ell$ does not divide either $a$ or $p_s$.

Let  $c = a/\ell$ and  $d = p_s^2$, and put
\begin{equation} \label{eqn}
f(x):=(dx^2+\ell x)\circ (x^2+cx) =
dx^4 + 2dc x^3 + (dc^2+\ell)x^2 + \ell c x,\end{equation}
so $f$ is decomposable over $K$.  Note that $f$ has coefficients in $R$,
since $ap_s/\ell$ lies in $R$.

Pick nonlinear $g,h\in K[x]$ such that $g\circ h=f$.  Let $\mu\in K[x]$
be a linear polynomial such that $\mu\circ h$ is monic and has no constant
term.  Then $f(x)=(g\circ\mu^{-1})\circ (\mu\circ h)$, and since $f(0)=0$
it follows that $g\circ\mu^{-1}$ has no constant term.  By inspecting
(\ref{eqn}), we see that the coefficients of $f$ uniquely determine
the coefficients of $g\circ\mu^{-1}$ and $\mu\circ h$, so
$g\circ\mu^{-1}=dx^2+\ell x$ and $\mu\circ h=x^2+cx$.
Writing $\mu=u^{-1}x+v$, it follows that there exist $u\in K^*$ and $v\in K$ such
that
\[
g = \frac{d}{u^2}x^2 + \frac{2dv+\ell}u x + (dv^2+\ell v) \qquad\text{ and }\qquad
h = ux^2 + ucx - uv.\]
If we can choose such $g$ and $h$ with coefficients in $R$, then
$R$ contains $\{u, uc, uv, d/u^2, (2dv+\ell)/u\}$, so
$R$ contains $\ell/u=(2dv+\ell)/u-2(uv)(d/u^2)$.
But $R$ contains $d/u^2=(p_s/u)^2$, so our hypothesis implies that $R$
contains $p_s/u$.  Thus $u$ divides both $\ell$ and $p_s$ (in $R$);
since $\ell$ and $p_s$ are non-associate
irreducibles, we must have $u\in R^*$.  Finally, since $uc\in R$, it follows
that $R$ contains $c=a/\ell$, contradicting the fact that $\ell\nmid a$.
Therefore $f\in R[x]$ is decomposable over $K$ but not over $R$.
\end{proof}

\begin{remark}
A positive answer to Question~\ref{q2} is provided via a different
argument in \cite[Prop.~2.6]{Turnwald}.
\end{remark}

\vspace{0.3cm}
We do not know how far Theorem~\ref{main2} can be generalized.  We pose the
following modification of Question~\ref{q2}:

\begin{question} \label{us2}
Let $R$ be an integral domain of characteristic zero which is not a unique
factorization domain, and let $K$ be a field containing $R$.  Does there exist
an indecomposable polynomial over $R$ which decomposes over $K$?
\end{question}


\section{Final note}

There is a mistake in \cite[Remark~1.2]{Gusic}, which attempts
to show that
if $K$ is a field of characteristic zero, and nonconstant $g,h,G,H\in K[x]$
satisfy $g\circ h=G\circ H$ and $\deg h=\deg H$, then there exist $a,b\in K$
such that $H=ah+b$.  The argument in \cite{Gusic} relies on an
incorrect assertion, of which a special case says that the sum of a quadratic
and cubic polynomial over $K$ cannot equal the sum of a linear and cubic
polynomial over $K$.  Since the strategy of the argument is novel,
we give here a corrected version of the proof (and we thank
I.~Gusi\'c for clarifying what was being attempted in \cite{Gusic}).

Write $H=ah+h_0$ with $a\in K$ and $\deg(h_0)<\deg(H)$.  We will show that
$h_0$ is a constant polynomial.  For, if $h_0\ne 0$ then Taylor expansion yields
\[
g\circ h = G \circ (ah+h_0) = \sum_{i=0}^m (G^{(i)}\circ h_0)\frac{(ah)^i}{i!},
\]
where $m:=\deg(G)$.  The left side is a $K$-linear combination of
powers of $h$, and the right side is the sum of polynomials
of degrees $(m-i)\deg(h_0)+i\deg(h)$ for $0\le i\le m$.  Moreover,
the polynomial of degree $m\deg(h)$ in the latter sum is $c h^m$ for some
$c\in K^*$.  After subtracting $c h^m$ from both sides, the right
side has degree $\deg(h_0)+(m-1)\deg(h)$, while the left side has
degree divisible by $\deg(h)$.  Thus $\deg(h)$ divides $\deg(h_0)$,
and since $0\le\deg(h_0)<\deg(H)=\deg(h)$ we conclude that $\deg(h_0)=0$.

We close by remarking that this result was first proved by Ritt \cite{Ritt}
in case $K=\C$, via Riemann surface techniques, and was later proved by Levi
\cite{Levi} by explicitly computing the coefficients of $g\circ h$
(see also \cite[Lemma~2.3]{GTZ}).  The result can also be proved by means of
formal Laurent series \cite{McConnell} or inertia groups \cite[Cor.~2.9]{ZM}.




\begin{thebibliography}{99}
\newcommand{\au}[1]{{#1},}
\newcommand{\ti}[1]{\textit{#1},}
\newcommand{\jo}[1]{{#1}}
\newcommand{\vo}[1]{\textbf{#1}}
\newcommand{\yr}[1]{(#1),}
\newcommand{\pp}[1]{#1.}
\newcommand{\ppa}[1]{#1,}
\newcommand{\pps}[1]{#1;}
\newcommand{\bk}[1]{{#1},}
\newcommand{\inbk}[1]{in: \bk{#1}}
\newcommand{\xxx}[1]{arXiv:#1.}

\bibitem{BWZ}
\au{R.~M. Beals, J.~L. Wetherell and M.~E. Zieve}
\ti{Polynomials with a common composite}
\jo{Israel J. Math.}
\vo{174}
\yr{2009}
\ppa{93--117}
\xxx{0707.1552}

\bibitem{BN}
\au{A.~F. Beardon and T.~W. Ng}
\ti{On Ritt's factorization of polynomials}
\jo{J. London Math. Soc.}
\vo{62}
\yr{2000}
\pp{127--138}

\bibitem{BT}
\au{Y.~F. Bilu and R.~F. Tichy}
\ti{The Diophantine equation $f(x)=g(y)$}
\jo{Acta Arith.}
\vo{95}
\yr{2000}
\pp{261--288}

\bibitem{BM}
\au{A. Bremner and P. Morton}
\ti{Polynomial relations in characteristic $p$}
\jo{Quart. J. Math. Oxford Ser. 2}
\vo{29}
\yr{1978}
\pp{335--347}

\bibitem{DW}
\au{F. Dorey and G. Whaples}
\ti{Prime and composite polynomials}
\jo{J. Algebra}
\vo{28}
\yr{1974}
\pp{88--101}

\bibitem{DG}
\au{A. Dujella and I. Gusi\'c}
\ti{Indecomposability of polynomials and related Diophantine equations}
\jo{Q. J. Math.}
\vo{57}
\yr{2006}
\pp{193--2001}

\bibitem{DT}
\au{A. Dujella and R. F. Tichy}
\ti{Diophantine equations for second-order recursive sequences of polynomials}
\jo{Q. J. Math.}
\vo{52}
\yr{2001}
\pp{161--169}

\bibitem{Engstrom}
\au{H.~T. Engstrom}
\ti{Polynomial substitutions}
\jo{Amer. J. Math.}
\vo{63}
\yr{1941}
\pp{249--255}

\bibitem{Fried-Ritt}
\au{M.~D. Fried}
\ti{On a theorem of Ritt and related Diophantine problems}
\jo{J. Reine Angew. Math.}
\vo{264}
\yr{1973}
\pp{40--55}

\bibitem{FriedMacRae}
\au{M.~D. Fried and R.~E. MacRae}
\ti{On the invariance of chains of fields}
\jo{Illinois J. Math.}
\vo{13}
\yr{1969}
\pp{165--171}

\bibitem{GTZ}
\au{D. Ghioca, T.~J. Tucker and M.~E. Zieve}
\ti{Intersections of polynomial orbits, and a dynamical Mordell-Lang
conjecture}
\jo{Invent. Math.}
\vo{171}
\yr{2008}
\ppa{463--483}
\xxx{0705.1954v2}

\bibitem{Gusic}
\au{I. Gusi\'c}
\ti{On decomposition of polynomials over rings}
\jo{Glas. Mat. Ser. III}
\vo{43} (63)
\yr{2008}
\pp{7--12}

\bibitem{Levi}
\au{H. Levi}
\ti{Composite polynomials with coefficients in an arbitrary
  field of characteristic zero}
\jo{Amer. J. Math.}
\vo{64}
\yr{1942}
\pp{389--400}

\bibitem{McConnell}
\au{A. McConnell}
\ti{Polynomial subfields of $k(x)$}
\jo{J. Reine Angew. Math.}
\vo{266}
\yr{1974}
\pp{136--139}

\bibitem{Ritt}
\au{J.~F. Ritt}
\ti{Prime and composite polynomials}
\jo{Trans. Amer. Math. Soc.}
\vo{23}
\yr{1922}
\pp{51--66}

\bibitem{Schinzel}
\au{A. Schinzel}
\bk{Polynomials with Special Regard to Reducibility}
Cambridge University Press, 2000.

\bibitem{Tortrat}
\au{P. Tortrat}
\ti{Sur la composition des polyn{\^o}mes}
\jo{Colloq. Math.}
\vo{55}
\yr{1988}
\pp{329--353}


\bibitem{Turnwald}
\au{G. Turnwald}
\ti{On Schur's conjecture}
\jo{J. Austral. Math. Soc. Ser. A}
\vo{58}
\yr{1995}
\pp{312--357}

\bibitem{Wegner}
\au{U. Wegner}
\ti{\"Uber die ganzzahligen Polynome, die f\"ur unendlich viele Primzahlmoduln
Permutationen liefern}
dissertation, Berlin, 1928.

\bibitem{Wyman}
\au{B. K. Wyman}
\ti{Polynomial decomposition over rings}
dissertation, Michigan, 2010.

\bibitem{Zannier}
\au{U. Zannier}
\ti{Ritt's second theorem in arbitrary characteristic}
\jo{J. Reine Angew. Math.}
\vo{445}
\yr{1993}
\pp{175--203}

\bibitem{ZM}
\au{M.~E. Zieve and P. Mueller}
\ti{On Ritt's polynomial decomposition theorems}
submitted for publication,
\xxx{0807.3578v1}

\end{thebibliography}
\end{document}